\documentclass[12pt]{amsart}

\usepackage{ucs}

\usepackage{amssymb}
\usepackage{amsthm}
\usepackage{amsmath}
\usepackage{latexsym}
\usepackage[cp1251]{inputenc}
\usepackage[mathcal]{eucal}
\usepackage{graphicx}
\usepackage{wrapfig}
\usepackage{caption}
\usepackage{subcaption}
\usepackage{indentfirst}
\usepackage[left=2.6cm,right=2.6cm,top=3cm,bottom=3cm,bindingoffset=0cm]{geometry}
\usepackage{enumerate}
\usepackage{setspace}

\DeclareMathOperator{\aut}{Aut}

\DeclareMathOperator{\cay}{Cay}
\DeclareMathOperator{\cyc}{Cyc}

\DeclareMathOperator{\iso}{Iso}

\DeclareMathOperator{\orb}{Orb}

\DeclareMathOperator{\rk}{rk}

\DeclareMathOperator{\Span}{Span}

\DeclareMathOperator{\sym}{Sym}
\DeclareMathOperator{\rad}{rad}

\DeclareMathOperator{\alg}{Alg}

\makeatletter 
\def\@seccntformat#1{\csname the#1\endcsname. } 
\def\@biblabel#1{#1.}

\makeatother

\title{On separable Schur rings over abelian groups}

\author{Grigory Ryabov}
\address{Sobolev Institute of Mathematics, Novosibirsk, Russia}
\address{Novosibirsk State University, Novosibirsk, Russia}
\email{gric2ryabov@gmail.com}
\thanks{The work is supported by the Russian Foundation for Basic Research (project 18-01-00752)}

\date{}

\newtheorem{prop}{Proposition}[section]
\newtheorem{theo}{Theorem}[section]

\newtheorem{lemm}[prop]{Lemma}

\theoremstyle{definition}

\begin{document}

\vspace{\baselineskip}
\vspace{\baselineskip}

\vspace{\baselineskip}

\vspace{\baselineskip}

\begin{abstract}
A finite group is said to be  \emph{weakly separable} if every algebraic isomorphism between two $S$-rings over this group is induced by a combinatorial isomorphism. In the paper we prove that every abelian weakly separable group  belongs to one of several explicitly given families only.
\end{abstract}

\maketitle

\section{Introduction}
A Schur ring or $S$-ring over a finite group $G$ can be defined as a subring of the group ring $\mathbb{Z}G$ that is a free $\mathbb{Z}$-module spanned by a partition of $G$ closed under taking inverse and containing the identity element of $G$ as a class (see Subsection~2.1 for the exact definition). The theory of $S$-rings was initiated by Schur~\cite{Schur} and later developed by  Wielandt~\cite{Wi} and his followers. 

Let $\mathcal{A}$ and $\mathcal{A}^{'}$ be $S$-rings over groups $G$ and $G^{'}$ respectively. An \emph{algebraic isomorphism} from $\mathcal{A}$ to $\mathcal{A}^{'}$ is defined to be a ring isomorphism of them. A \emph{(combinatorial) isomorphism} from $\mathcal{A}$ to $\mathcal{A}^{'}$ is defined to be an isomorphism of the corresponding Cayley schemes (see Subsection~2.2). One can check that every combinatorial isomorphism induces the algebraic one. However the converse statement is not true, see e.g.~\cite{EP1}.  

Let $\mathcal{K}$ be a class of groups. Following~\cite{EP3}, we say that an $S$-ring $\mathcal{A}$ is \emph{separable} with respect to $\mathcal{K}$ if every algebraic isomorphism from $\mathcal{A}$ to an $S$-ring over a group from $\mathcal{K}$ is induced by a combinatorial one. Note that if $\mathcal{A}$ is separable with respect to $\mathcal{K}$ then $\mathcal{A}$ is determined up to isomorphism in the class of $S$-rings over groups from $\mathcal{K}$ by the tensor of its structure constants (with respect to the basis corresponding to the partition of the underlying group). So the question when an $S$-ring is separable is a particular case of the following general question arising in different parts of combinatorics:  when a combinatorial structure is determined up to isomorphism  by its parameters? For more details see~\cite{EP3,CP}.

A finite group $G$ is said to be \emph{separable} with respect to $\mathcal{K}$ if every $S$-ring  over $G$ is separable with respect to $\mathcal{K}$ (see~\cite{Ry2}). Denote by $\mathcal{K}_A$ and $\mathcal{K}_G$ the classes of all abelian groups and all groups isomorphic to a given group $G$ respectively. We say that a group $G$ is \emph{weakly separable} if it is separable with respect to $\mathcal{K}_G$. Clearly,  if  $G$ is abelian and separable with respect to  $\mathcal{K}_A$ then it is weakly separable. If $G$ is weakly separable then  the isomorphism problem for Cayley graphs over~$G$ can be solved efficiently by using the Weisfeiler-Leman algorithm~\cite{WeisL}. In the sense of~\cite{KPS} this means that the Weifeiler-Leman dimension of the class of Cayley graphs over~$G$ is at most~3. For more information about separable $S$-rings and groups we refer the reader to~\cite{EP2,Ry1}.  

Few results on separable groups are known. Cyclic $p$-groups are separable with respect to $\mathcal{K}_{A}$ (\cite{EP3,Ry3}). Denote the cyclic group of order~$n$ and the elementary abelian group of order $p^k$, where $p$ is a prime and $k\geq 0$, by $C_n$ and $E_{p^k}$ respectively. The groups $C_{p}\times C_{p^k}$ and $E_{p^3}$ are separable with respect to $\mathcal{K}_{A}$ for $p\in\{2,3\}$ and $k\geq 1$ (see~\cite{Ry1,Ry3}). A complete classification of abelian $p$-groups, which are separable with respect to $\mathcal{K}_{A}$, was obtained in~\cite{Ry3}. Namely, in~\cite{Ry3} it was proved that an abelian $p$-group is  separable with respect to $\mathcal{K}_{A}$ if and only if it is cyclic or isomorphic to one of the above mentioned groups. In~\cite{Ry2} it was proved that the group $E_4\times C_p$ is separable with respect to $\mathcal{K}_{A}$ for every prime $p$. 

The main results of the paper are given in the following three theorems.

\begin{theo}\label{main1}
If a cyclic group of order $n$ is weakly separable then $n$ belongs to one of the following families of integers:
$$p^k, pq^k, 2pq^k, pqr, 2pqr,$$
where $p,q,r$ are distinct primes and $k\geq 0$ is an integer. Moreover, cyclic $p$-groups are separable with respect to $\mathcal{K}_{A}$.
\end{theo}

In fact, the next result is a direct corollary  of~\cite[Theorem~1.2]{Ry3}.

\begin{theo}\label{main2}
An elementary abelian non-cyclic group of order $n$ is weakly separable if and only if $n\in\{4,8,9,27\}$. Moreover, if $n\in\{4,8,9,27\}$ then the elementary abelian group of order $n$ is separable with respect to $\mathcal{K}_{A}$.
\end{theo}

\begin{theo}\label{main3}
An abelian weakly separable group, which is neither cyclic nor elementary abelian, is isomorphic to a group from one of the following eight families:

$(1)$ $C_2 \times C_{2^k}$, $C_{2p}\times C_{2^k}$, $E_4 \times C_{p^k}$, $E_4 \times C_{pq}$,

$(2)$ $C_3 \times C_{3^k}$, $C_{6}\times C_{3^k}$, $E_9\times C_{q}$, $E_9 \times C_{2q}$,

\noindent where $p$ and $q$ are distinct primes, $p\neq 2$, and $k\geq 1$ is an integer. Moreover, the groups $C_2 \times C_{2^k}$,  $C_3 \times C_{3^k}$, and $E_4 \times C_{p}$ are separable with respect to $\mathcal{K}_{A}$.
\end{theo}

We do not know whether all groups from Theorem~\ref{main1}  and Theorem~\ref{main3} are weakly separable. This question seems to be quite difficult. Indeed, to prove that a given  group is weakly separable it is required to check that every $S$-ring over this group is separable whereas to prove that a given  group is not weakly separable it is sufficient to find at least one non-separable $S$-ring over this group.

Another important problem concerned with $S$-rings is the problem of determining of all Schur groups suggested in~\cite{Po}. Recall that a finite group is called a \emph{Schur} group if every $S$-ring over this group is \emph{schurian}, i.e. it arises from a suitable permutation group. All cyclic Schur groups were classified in~\cite[Theorem~1.1]{EKP1}. From this result and Theorem~\ref{main1} it follows that every weakly separable cyclic group is Schur. All elementary abelian Schur groups were classified in~\cite[Theorem~1.2]{EKP2}. This result and Theorem~\ref{main2} imply that every  weakly separable elementary abelian group is Schur and  $E_{16}$ and $E_{32}$ are the only elementary abelian Schur groups which are not weakly separable. Due to~\cite[Theorem~1.3]{EKP2}, every abelian Schur group, which is neither cyclic nor elementary abelian, belongs to one of nine explicitly given families only. The list of these families includes each of eight families from Theorem~\ref{main3} and the groups $E_{16}\times C_p$, where $p$ is a prime, which are not weakly separable.

The main tool of the proofs of Theorem~\ref{main1} and Theorem~\ref{main3} is a sufficient condition for an abelian group to be non-weakly separable (Proposition~\ref{suff}). In the proof of Proposition~\ref{suff} we construct a schurian $S$-ring $\mathcal{A}$ over an abelian group and an algebraic isomorphism $\varphi$ from $\mathcal{A}$ to itself such that the algebraic fusion of $\mathcal{A}$ with respect to $\langle \varphi \rangle$ (see Subsection~2.3 for definitions) coincides with non-schurian $S$-ring constructed in~\cite[Theorem~4.1]{EKP2}. Lemma~\ref{fusion} implies that $\varphi$ is not induced by an isomorphism and hence $\mathcal{A}$ is not separable.

To make the paper self-contained we collect the basic facts on $S$-rings in Section~2.

The author would like to thank prof. I. Ponomarenko for the fruitful discussions on the subject matters.

~\

{\bf Notation.}

The ring of rational integers is denoted by $\mathbb{Z}$.

Let  $G$ be a finite group and $X\subseteq G$. The element $\sum_{x\in X} {x}$ of the group ring $\mathbb{Z}G$ is denoted by $\underline{X}$.

The order of $g\in G$ is denoted by $|g|$.

The set $\{x^{-1}:x\in X\}$ is denoted by $X^{-1}$.

The subgroup of $G$ generated by $X$ is denoted by $\langle X\rangle$; we also set $\rad(X)=\{g\in G:\ gX=Xg=X\}$.

Given a set $X\subseteq G$ the set $\{(g,xg): x\in  X, g\in G\}$ of arcs of the Cayley graph $\cay(G,X)$ is denoted by $R(X)$.

The group of all permutations of a set $\Omega$ is denoted by $\sym(\Omega)$.

The subgroup of $\sym(G)$ induced by right multiplications of $G$ is denoted by $G_{right}$.

If a group $K$ acts on a set $\Omega$ then the set of all orbtis of $K$ on $\Omega$ is denoted by $\orb(K,\Omega)$.

If $K\leq \sym(\Omega)$ and $\alpha\in \Omega$ then the stabilizer of $\alpha$ in $K$ is denoted by $K_{\alpha}$.

If $G$ is a finite abelian group and $p$ is a prime divisor of $|G|$ then the Sylow $p$-subgroup of $G$ is denoted by $G_p$.

If $n$ is an integer then the number of prime divisors of $n$ and the total number of prime divisors of $n$ (with multiplicity) are denoted by $\omega(n)$ and $\Omega(n)$ respectively.

The cyclic group of order $n$ is denoted by  $C_n$.

The elementary abelian group of order $p^k$ is denoted by $E_{p^k}$.

\section{$S$-rings}
In what follows, we use the notation and terminology of~\cite{MP,Ry1}.

\subsection{Definitions}

Let $G$ be a finite group and $\mathbb{Z}G$  the integral group ring. The identity element of $G$ is denoted by $e$.  A subring  $\mathcal{A}\subseteq \mathbb{Z} G$ is called an \emph{$S$-ring} over $G$ if there exists a partition $\mathcal{S}=\mathcal{S}(\mathcal{A})$ of~$G$ such that:

$(1)$ $\{e\}\in\mathcal{S}$,

$(2)$  if $X\in\mathcal{S}$ then $X^{-1}\in\mathcal{S}$,

$(3)$ $\mathcal{A}=\Span_{\mathbb{Z}}\{\underline{X}:\ X\in\mathcal{S}\}$.

The elements of $\mathcal{S}$ are called the \emph{basic sets} of  $\mathcal{A}$. The number $|\mathcal{S}|$ is called the \emph{rank} of~$\mathcal{A}$ and denoted by $\rk(\mathcal{A})$. Given $X,Y,Z\in\mathcal{S}$ the number of distinct representations of $z\in Z$ in the form $z=xy$ with $x\in X$ and $y\in Y$ is denoted by $c^Z_{X,Y}$. If $X,Y\in\mathcal{S}$ then 
$$\underline{X}~\underline{Y}=\sum_{Z\in \mathcal{S}(\mathcal{A})}c^Z_{X,Y}\underline{Z}.$$ 
This means that the numbers  $c^Z_{X,Y}$ are the structure constants of $\mathcal{A}$ with respect to the basis $\{\underline{X}:\ X\in\mathcal{S}\}$. One can check that 
$$|Z|c^{Z^{-1}}_{X,Y}=|X|c^{X^{-1}}_{Y,Z}=|Y|c^{Y^{-1}}_{Z,X}~\eqno(1)$$
for all $X,Y,Z\in \mathcal{S}(\mathcal{A})$.

A set $X \subseteq G$ is called an \emph{$\mathcal{A}$-set} if $\underline{X}\in \mathcal{A}$. A subgroup $H \leq G$ is called an \emph{$\mathcal{A}$-subgroup} if $H$ is an $\mathcal{A}$-set. With each $\mathcal{A}$-set $X$ one can naturally associate two $\mathcal{A}$-subgroups, namely $\langle X \rangle$ and $\rad(X)$.

A section $U/L$ of $G$ is said to be an \emph{$\mathcal{A}$-section} if $U$ and $L$ are $\mathcal{A}$-subgroups. If $S=U/L$ is an $\mathcal{A}$-section then the module
$$\mathcal{A}_S=Span_{\mathbb{Z}}\left\{\underline{X}^{\pi}:~X\in\mathcal{S}(\mathcal{A}),~X\subseteq U\right\},$$
where $\pi:U\rightarrow U/L$ is the canonical epimorphism, is an $S$-ring over $S$.

If $K \leq \aut(G)$ then the set  $\orb(K,G)$ forms a partition of  $G$ that defines an  $S$-ring $\mathcal{A}$ over $G$.  In this case  $\mathcal{A}$ is called \emph{cyclotomic} and denoted by $\cyc(K,G)$.

Let $S=U/L$ be an $\mathcal{A}$-section. The $S$-ring~$\mathcal{A}$ is called the \emph{$S$-wreath product}  if $L\trianglelefteq G$ and $L\leq\rad(X)$ for all basic sets $X$ outside~$U$. The $S$-wreath product is called \emph{nontrivial} or \emph{proper}  if $e\neq L$ and $U\neq G$.

\subsection{Isomorphisms and schurity}

Let $\mathcal{A}$ and $\mathcal{A}^{'}$ be $S$-rings over groups $G$ and $G^{'}$ respectively. A bijection $f:G\rightarrow G^{'}$ is defined to be \emph{a (combinatorial) isomorphism } from $\mathcal{A}$ over to $\mathcal{A}^{'}$  if 
$$\{R(X)^f: X\in \mathcal{S}(\mathcal{A})\}=\{R(X^{'}): X^{'}\in \mathcal{S}(\mathcal{A}^{'})\},$$ where $R(X)^f=\{(g^f,~h^f):~(g,~h)\in R(X)\}$. The group $\iso(\mathcal{A})$ of all isomorphisms from $\mathcal{A}$ onto itself has a normal subgroup
$$\aut(\mathcal{A})=\{f\in \iso(\mathcal{A}): R(X)^f=R(X)~\text{for every}~X\in \mathcal{S}(\mathcal{A})\}.$$
This subgroup is called the \emph{automorphism group} of $\mathcal{A}$. Note that $\aut(\mathcal{A})\geq G_{right}$. 

Let $K$ be a subgroup of $\sym(G)$ containing  $G_{right}$. In~\cite{Schur} Schur proved that the $\mathbb{Z}$-submodule
$$V(K,G)=\Span_{\mathbb{Z}}\{\underline{X}:~X\in \orb(K_e,~G)\},$$
is an $S$-ring over $G$. An $S$-ring $\mathcal{A}$ over  $G$ is called \emph{schurian} if $\mathcal{A}=V(K,G)$ for some $K$ such that $G_{right}\leq K \leq \sym(G)$. In fact, there is a lot of non-shurian $S$-rings. An infinite family of them can be found in~\cite{Wi}. Every cyclotomic $S$-ring is schurian. Indeed, if $\mathcal{A}=\cyc(K,G)$ for some $K\leq \aut(G)$ then $\mathcal{A}=V(G_{right}\rtimes K,G)$.

\subsection{Algebraic isomorphisms and separability}

 A bijection $\varphi:\mathcal{S}(\mathcal{A})\rightarrow\mathcal{S}(\mathcal{A}^{'})$ is defined to be an \emph{algebraic isomorphism} from $\mathcal{A}$  to $\mathcal{A}^{'}$ if 
$$c_{X,Y}^Z=c_{X^{\varphi},Y^{\varphi}}^{Z^{\varphi}}$$ 
for all $X,Y,Z\in \mathcal{S}(\mathcal{A})$. The mapping $\underline{X}\rightarrow \underline{X}^{\varphi}$ is extended by linearity to the ring isomorphism of $\mathcal{A}$  and $\mathcal{A}^{'}$.   An algebraic isomorphism from $\mathcal{A}$ to itself is called an \emph{algebraic automorphism} of $\mathcal{A}$. The group of all algebraic automorphisms of $\mathcal{A}$ is denoted by $\aut_{\alg}(\mathcal{A})$. 

Every isomorphism $f$ of $S$-rings  preserves  the structure constants  and hence $f$ induces the algebraic isomorphism $\varphi_f$. However, not every algebraic isomorphism is induced by a combinatorial one, see, e.g.,~\cite{EP1}. Let $\mathcal{K}$ be a class of groups. An $S$-ring $\mathcal{A}$ is defined to be \emph{separable} with respect to $\mathcal{K}$ if every  algebraic isomorphism from $\mathcal{A}$ to an $S$-ring over a group from $\mathcal{K}$ is induced by a combinatorial isomorphism. For every group $G$ the $S$-ring of rank~2 over $G$ and $\mathbb{Z}G$ are separable with respect to the class of all finite groups.

A finite group $G$ is said to be \emph{separable} with respect to $\mathcal{K}$ if every $S$-ring  over $G$ is separable with respect to $\mathcal{K}$. 

We say that $G$ is \emph{weakly separable} if it is separable with respect to the class $\mathcal{K}_G$ of groups isomorphic to $G$.

\begin{lemm}\cite[Lemma~2.5]{Ry3}\label{nonsepwr}
Let $H$ be a normal subgroup of a group $G$ and $\mathcal{B}$  an $S$-ring over $H$. Suppose that $\varphi \in \aut_{\alg}(\mathcal{B})$ and $\varphi$ is not induced by an isomorphism. Then there exists $\psi\in \aut_{\alg}(\mathcal{A})$, where $\mathcal{A}=\mathcal{B}\wr \mathbb{Z}(G/H)$, such that $\psi^H=\varphi$ and  $\psi$ is not induced by an isomorphism.
\end{lemm}

\begin{lemm}\label{nonsep}
For every group $G$ of order at least~4 there exist an $S$-ring $\mathcal{A}$ over $G\times G$ and $\varphi\in \aut_{\alg}(\mathcal{A})$ such that $\varphi$ is not induced by an isomorphism.
\end{lemm}

\begin{proof}
Let $G$ be a group of order at least~4. In~\cite[pp.90-92]{CP} it was proved that for the Cayley scheme $\mathcal{C}$ over $G\times G$ constructed in~\cite[Theorem~3.3]{Klin} there exists an algebraic isomorphism from $\mathcal{C}$ to itself which is not induced by an isomorphism. This implies that for the $S$-ring $\mathcal{A}$ over $G\times G$ corresponding to $\mathcal{C}$ there exists an algebraic isomorphism from $\mathcal{A}$ to itself which is not induced by an isomorphism and the lemma is proved.
\end{proof}

\subsection{Algebraic fusions}

Let $\mathcal{A}$ be an $S$-ring over $G$ and $\Phi\leq \aut_{\alg}(\mathcal{A})$. Given $X\in \mathcal{S}(\mathcal{A})$ put $X^{\Phi}=\bigcup \limits_{\varphi\in \Phi} X^{\varphi}$. The partition
$$\{X^{\Phi}: X\in \mathcal{S}(\mathcal{A})\}$$
defines an $S$-ring over $G$ called the \emph{algebraic fusion} of $\mathcal{A}$ with respect to $\Phi$ and denoted by $\mathcal{A}^{\Phi}$.

\begin{lemm}\label{fusion}
Let $\mathcal{A}$ be a schurian $S$-ring over $G$ and $K\leq \iso(\mathcal{A})$. Then $\mathcal{A}^{\Phi}$, where $\Phi=\{\varphi_f: f\in K\}$, is also schurian.
\end{lemm}

\begin{proof}
Follows from~\cite[Proposition~2.3.28]{CP}. See also~\cite[Lemma~2.1]{Ry3}.
\end{proof}

\section{A sufficient condition of non-separability}

Following~\cite{EKP2}, set $\Omega^{*}(n)=\Omega(n)$ whenever $n$ is odd and $\Omega^{*}(n)=\Omega(n/2)$ whenever $n$ is even.

\begin{prop}\label{suff}
Let $G_i$ be an abelian group and $\Omega^{*}(|G_i|)\geq 2$, $i=1,2$. Then $G=G_1\times G_2$ is not weakly separable.
\end{prop}

\begin{proof}
Since $G_i$ is abelian and $\Omega(|G_i|)^{*}\geq 2$, there exists $H_i\leq G_i$ such that 
$$|H_i|=p_iq_i~\text{\quad or \quad}~|H_i|=4q_i~\text{\quad or \quad}~|H_i|=8,$$ 
where $p_i,q_i$ are odd primes, $i\in\{1,2\}$.  

If $|H_1|$ and $|H_2|$ are divisible by~4 then $H_1\times H_2$ contains a subgroup $U$ isomorphic to one of the following groups: 
$$C_4\times C_4,~E_4\times C_4,~E_4\times E_4.$$
There exists an $S$-ring $\mathcal{A}$ over $U$ and an algebraic isomorphism from $\mathcal{A}$ to itself that is not induced by an isomorphism. Indeed, for $U\cong C_4\times C_4$ and for $U\cong E_4\times E_4$ this follows from Lemma~\ref{nonsep} and for $U\cong E_4\times C_4$ this follows from~\cite[pp.10-11]{Ry3}. So $G=G_1\times G_2$ is not weakly separable by Lemma~\ref{nonsepwr}. Further we may assume that $|H_1|=p_1q_1$, where $p_1,q_1$ are odd primes. 

In view of Lemma~\ref{nonsepwr}, to prove the proposition it is sufficient to construct an $S$-ring $\mathcal{A}$ over $H_1\times H_2$ and $\varphi\in \aut_{\alg}(\mathcal{A})$ which is not induced by an isomorphism. Denote the subgroup of $H_1$ of order $p_1$ and its generator by $A_1$ and  $a_1$ respectively. If $|H_2|$ is divisible by~4 then denote the subgroup of $H_2$ of order~4 by $A_2$; otherwise denote the subgroup of $H_2$ of order $p_2$ by $A_2$. If $A_2$ is cyclic then denote its generator by $a_2$; if $A_2\cong E_4$ then denote its generators by $a_{21}$ and $a_{22}$. Put $L=A_1\times A_2$. Let $b_1\in H_1\setminus A_1$ and $b_2\in H_2\setminus A_2$.  Clearly, $\pi_1(b_1)$ and $\pi_2(b_2)$, where $\pi_1:H_1\rightarrow H_1/A_1$ and  $\pi_2:H_2\rightarrow H_2/A_2$ are the canonical epimorphisms, generate $H_1/A_1$ and $H_2/A_2$ respectively. If $A_2$ is cyclic then define $\sigma_1,\sigma_2,\sigma_3\in \aut(H_1\times H_2)$ in the following way:
$$\sigma_1: (a_1,a_2,b_1,b_2)\rightarrow (a_1^{-1},a_2^{-1},b_1^{-1},b_2^{-1}),$$
$$\sigma_2: (a_1,a_2,b_1,b_2)\rightarrow (a_1,a_2,b_1a_1,b_2),$$
$$\sigma_3: (a_1,a_2,b_1,b_2)\rightarrow (a_1,a_2,b_1,b_2a_2).$$
If $A_2\cong E_4$ then  define $\sigma_1,\sigma_2,\sigma_3,\sigma_4\in \aut(H_1\times H_2)$ in the following way:
$$\sigma_1: (a_1,a_{21},a_{22},b_1,b_2)\rightarrow (a_1^{-1},a_{22},a_{21},b_1^{-1},b_2^{-1}),$$
$$\sigma_2: (a_1,a_{21},a_{22},b_1,b_2)\rightarrow (a_1,a_{21},a_{22},b_1a_1,b_2),$$
$$\sigma_3: (a_1,a_{21},a_{22},b_1,b_2)\rightarrow (a_1,a_{21},a_{22},b_1,b_2a_{21}),$$
$$\sigma_4: (a_1,a_{21},a_{22},b_1,b_2)\rightarrow (a_1,a_{21},a_{22},b_1,b_2a_{22}).$$
The straightforward check shows that in both cases all $\sigma_i$ pairwise commute. Put 
$$K= 
 \begin{cases}
   \langle \sigma_1 \rangle \times \langle \sigma_2 \rangle \times \langle \sigma_3 \rangle~\text{if}~A_2~\text{is cyclic}\\
  \langle \sigma_1 \rangle \times \langle \sigma_2 \rangle \times \langle \sigma_3 \rangle\times \langle \sigma_4 \rangle~\text{if}~A_2\cong E_4
 \end{cases}.$$
Let $\mathcal{A}=\cyc(K,H_1\times H_2)$. The definition of $K$ yields that the basic sets of $\mathcal{A}$ are the following:
$$X_{ij}=\{a_1^ia_2^j,a_1^{-i}a_2^{-j}\},~i=0,\ldots, |A_1|-1,~j=0,\ldots,|A_2|-1,$$
$$Y_{ij}=b_1^ia_2^jA_1 \cup b_1^{-i}a_2^{-j}A_1,~i=1,\ldots,|H_1/A_1|,~j=0,\ldots,|A_2|-1,$$
$$Z_{ij}=b_2^ia_1^jA_2 \cup b_2^{-i}a_1^{-j}A_2,~i=1,\ldots,|H_2/A_2|,~j=0,\ldots,|A_1|-1,$$
$$T_{ij}=b_1^ib_2^j(A_1\times A_2) \cup b_1^{-i}b_2^{-j}(A_1\times A_2),~i=1,\ldots,|H_1/A_1|,~j=1,\ldots,|H_2/A_2|,$$
if $A_2$ is cyclic and
$$X_{ij}=\{a_1^ia_{21}^j,a_1^{-i}a_{22}^{j}\},~X_i=a_{21}a_{22}\{a_1^i,a_1^{-i}\},~i=0,\ldots, |A_1|-1,~j=0,1,$$
$$Y_{ij}=b_1^ia_{21}^jA_1 \cup b_1^{-i}a_{22}^{j}A_1,~Y_{i}=a_{21}a_{22}(b_1^iA_1 \cup b_1^{-i}A_1)~i=1,\ldots,|H_1/A_1|,~j=0,1,$$
$$Z_{ij}=b_2^ia_1^jA_2 \cup b_2^{-i}a_1^{-j}A_2,~i=1,\ldots,|H_2/A_2|,~j=0,\ldots,|A_1|-1,$$
$$T_{ij}=b_1^ib_2^j(A_1\times A_2) \cup b_1^{-i}b_2^{-j}(A_1\times A_2),~i=1,\ldots,|H_1/A_1|,~j=1,\ldots,|H_2/A_2|,$$
if  $A_2\cong E_4$. Note that $\mathcal{A}$ is the $(H_1\times A_2)/A_2$-wreath product and $$X=X^{-1}~\eqno(2)$$ 
for every $X\in\mathcal{S}(\mathcal{A})$.

Let $\mathcal{W}=\{X_{ij}:~i\geq 1,~j\geq 1\}$. Define a permutation $\varphi$ on the set $\mathcal{S}(\mathcal{A})$ as follows:
$$X^{\varphi}= 
 \begin{cases}
   X_{(|A_1|-i)j}~\text{if}~X=X_{ij}~\text{for some}~i\geq 1~\text{and}~j\geq 1\\
   X~\text{if}~X\in \mathcal{S}(\mathcal{A})\setminus \mathcal{W}
 \end{cases}.$$
Note that $|\varphi|=2$ and 
$$|X^{\varphi}|=|X|=2~\eqno(3)$$
for every $X\in\mathcal{S}(\mathcal{A})$. Let us prove that $\varphi\in \aut_{\alg}(\mathcal{A})$. To do this we need to check that 
$$c^{Z^{\varphi}}_{X^{\varphi},Y^{\varphi}}=c^Z_{X,Y}$$
for all $X,Y,Z\in \mathcal{S}(\mathcal{A})$. If $X,Y,Z\notin \mathcal{W}$ then $X^{\varphi}=X$, $Y^{\varphi}=Y$, and $Z^{\varphi}=Z$. So $c^{Z^{\varphi}}_{X^{\varphi},Y^{\varphi}}=c^Z_{X,Y}$. Suppose that one of the sets among $X,Y,Z$ lies in $\mathcal{W}$. In view of~(1), (2), and (3) we may assume that $Z\in \mathcal{W}$. The straightforward computation shows that the elements $\underline{X_{ij}}$ and $\underline{X_{ij}}^{\varphi}=\underline{X_{(|A_1|-i)j}}$ enter the product $\underline{X}\underline{Y}$ with the same coefficients for all $i,j\geq 1$ and all $X,Y\in \mathcal{S}(\mathcal{A})\setminus \mathcal{W}$.  Therefore $c^{Z^{\varphi}}_{X^{\varphi},Y^{\varphi}}=c^{Z^{\varphi}}_{X,Y}=c^Z_{X,Y}$.

Suppose that at least two of the sets among $X,Y,Z$ lie in $\mathcal{W}$. Due to~(1), (2), and (3) we may assume that $X,Y\in \mathcal{W}$. A direct check shows that 
$$\underline{X}\underline{Y}= 
 \begin{cases}
   2e+\underline{T}~\text{if}~X=Y\\
   \underline{T_1}+\underline{T_2}~\text{if}~X\neq Y
 \end{cases}~\text{and \quad}~
\underline{X}^{\varphi}\underline{Y}^{\varphi}= 
 \begin{cases}
   2e+\underline{T}^{\varphi}~\text{if}~X=Y\\
   \underline{T_1}^{\varphi}+\underline{T_2}^{\varphi}~\text{if}~X\neq Y
 \end{cases}$$
for some $T,T_1,T_2\in \mathcal{S}(\mathcal{A})$ if $A_2$ is cyclic and
$$\underline{X}\underline{Y}=\underline{X}^{\varphi}\underline{Y}^{\varphi}$$
if $A_2\cong E_4$. Therefore $c^{Z^{\varphi}}_{X^{\varphi},Y^{\varphi}}=c^Z_{X,Y}$. Thus, $\varphi$ preserves the structure constants and hence $\varphi\in \aut_{\alg}(\mathcal{A})$.

Suppose that $\varphi$ is induced by an isomorphism. Since $\mathcal{A}$ is cyclotomic, it is schurian. Therefore $\mathcal{A}^{\langle \varphi \rangle}$ is also schurian by Lemma~\ref{fusion}. However, $\mathcal{A}^{\langle \varphi \rangle}$ coincides with  non-schurian $S$-ring constructed in~\cite[pp.8-11]{EKP2}, a contradiction. Thus, $\varphi$ is not induced by an isomorphism and the proposition is proved.
\end{proof}

\section{Proofs of Theorems~\ref{main1}, \ref{main2}, and~\ref{main3}}

\begin{proof}[Proof of Theorem~\ref{main1}]
The second part of the theorem follows from~\cite[Theorem~1.1]{Ry3}. Prove the first part of the theorem. Let $G$ be a cyclic group of order~$n$ and $G$  weakly separable. From Proposition~\ref{suff} it follows that $\omega(n)\leq 4$. If $\omega(n)=1$ then $n=p^k$ for some prime $p$ and integer $k\geq 1$ and $n$ belongs to the first family. 

Let $\omega(n)=2$. Then 
$$n=p^lq^k~\text{and}~G\cong C_{p^k}\times C_{q^l}$$ 
for some distinct primes $p,q$ and integers $k,l\geq 1$. If $2\notin\{p,q\}$ then $k=1$ or $l=1$. Indeed, if $k>1$ and $l>1$ then $\Omega^{*}(p^k)\geq 2$ and $\Omega^{*}(q^l)\geq 2$ and $G$ is not weakly separable by Proposition~\ref{suff}. So $n$ belongs to the second family. Suppose that $2\in\{p,q\}$. Without loss of generality we may assume that $p=2$. In this case $k\leq 2$ or $l=1$ because otherwise $G$ is not weakly separable by Proposition~\ref{suff}. This yields that $n=2q^l$ or $n=4q^l$ or $n=2^kq$. In the first and third cases $n$ belongs to the second family; in the second case  $n$ belongs to the third family. 

Let $\omega(n)=3$. Then 
$$n=p^kq^lr^m~\text{and}~G\cong C_{p^k}\times C_{q^l}\times C_{r^m}$$ 
for some distinct primes $p,q,r$ and integers $k,l,m\geq 1$. If $2\notin\{p,q,r\}$ then $k=l=m=1$ because otherwise $G$ is not weakly separable by Proposition~\ref{suff}. Therefore $n$ belongs to the fourth family. Suppose that $2\in\{p,q,r\}$. Without loss of generality we may assume that $p=2$. Proposition~\ref{suff} implies that $l=1$ or $m=1$. We may assume by symmetry that $l=1$. If $m\geq 2$ then in view of Proposition~\ref{suff}, we conclude that $k=1$ and $n=2qr^m$ belongs to the third family. If $m=1$ then $k\leq 2$ and hence $n=2qr$ or $n=4qr$. In the former case $n$ belongs to the third family; in the latter case $n$ belongs to fifth family.

Finally, let $\omega(n)=4$. Then 
$$n=p^kq^lr^mt^s~\text{and}~G\cong C_{p^k}\times C_{q^l}\times C_{r^m}\times C_{t^s}$$ 
for some distinct primes $p,q,r,t$ and integers $k,l,m,s\geq 1$. In this case $2\in\{p,q,r,t\}$. Indeed, if $2\notin\{p,q,r,t\}$ then $\Omega^{*}(p^kq^l)\geq 2$ and $\Omega^{*}(r^mt^s)\geq 2$ and $G$ is not weakly separable by Proposition~\ref{suff}. Without loss of generality we may assume that $p=2$. Due to Proposition~\ref{suff}, we have $k=l=m=s=1$ and hence $n=2qrt$. Therefore $n$ belongs to the fifth family. The theorem is proved.
\end{proof}

\begin{proof}[Proof of Theorem~\ref{main2}]
The elementary abelian groups of orders~4, 8, 9, 27 are separable with respect to $\mathcal{K}_A$ by~\cite[Theorem~1.2]{Ry3}. The necessity of the theorem follows from Lemma~\ref{nonsepwr} and Lemma~\ref{nonsep}.
\end{proof}

Before we prove Theorem~\ref{main3}, we provide an auxiliary lemma described Sylow subgroups of an abelian weakly separable group.

\begin{lemm}\label{sylow}
Let $G$ be an abelian weakly separable group and $p$ a prime divisor of $|G|$. If $p\in\{2,3\}$ then $G_p$ is isomorphic to one of the groups $C_{p^k}$, $C_p\times C_{p^k}$, $C_p^3$, where $k\geq 1$. If $p\geq 5$ then $G_p$  is cyclic.
\end{lemm}

\begin{proof}
Let $p\in\{2,3\}$. Assume that $G_p$ is the direct product of at least four cyclic groups. Then $G_p$ contains a subgroup isomorphic to $C_p^4$. In this case $G$ is not weakly separable by Lemma~\ref{nonsepwr} and Lemma~\ref{nonsep}, a contradiction. So $G_p$ is isomorphic to $C_{p^k}\times C_{p^l}\times C_{p^m}$ for some $k,l,m\geq 0$. At least two numbers among $k,l,m$ are smaller than $2$ because otherwise $G_p$ contains a subgroup isomorphic to $C_{p^2}\times C_{p^2}$ and $G$ is not weakly separable by Lemma~\ref{nonsepwr} and Lemma~\ref{nonsep}. Without loss of generality we may assume that $l\leq 1$ and $m\leq 1$. If $l=m=0$ then $G_p\cong C_{p^k}$. If $l=0$, $m=1$ or $l=1$, $m=0$ then $G_p\cong C_p\times C_{p^k}$. If $l=m=1$ then $k=1$ by~\cite[pp.10-11]{Ry3} and hence $G_p\cong C_p^3$. 

Now let $p\geq 5$. Assume that $G_p$ is non-cyclic. Then it contains a subgroup isomorphic to $C_p\times C_p$. In this case $G$ is not weakly separable by Lemma~\ref{nonsepwr} and Lemma~\ref{nonsep}, a contradiction. Therefore $G_p$ is cyclic. The lemma is proved.
\end{proof}

\begin{proof}[Proof of Theorem~\ref{main3}]
The second part of the theorem follows from~\cite[Theorem~1]{Ry1} and~\cite[Theorem~1]{Ry2}. Prove the first part of the theorem. Let $G$ be an abelian group of order $n$ and $G$ is  weakly separable. Suppose that $G$ is  neither cyclic nor elementary abelian. Since $G$ is non-cyclic, there exists a prime $p$ such that $G_p$ is non-cyclic. From Lemma~\ref{sylow} it follows that  
$$p\in\{2,3\}~\eqno(4)$$
and $G_p$ is isomorphic to one of the groups
$$C_p\times C_{p^k},C_p^3,~\eqno(5)$$
where $k\geq 1$.

Assume that $\omega(n)\geq 4$. Then $G$ contains a subgroup isomorphic to $C_p\times C_p\times C_q\times C_r\times C_t$ for some distinct primes $q,r,t$. In this case $G$ is not weakly separable by Proposition~\ref{suff}. Therefore $\omega(n)\leq 3$. If $\omega(n)=1$  then $G$ is a $p$-group and $G$ belongs to the first or fifth family by~Lemma~\ref{sylow}. 

Let $\omega(n)=2$. Then there exists a prime $q\neq p$ such that 
$$G=G_p\times G_q.$$ 
Due to~(4), we have $p\in\{2,3\}$. Suppose that $p=2$. If $|G_q|\geq q^2$ then $\Omega^{*}(|G_q|)\geq 2$. Proposition~\ref{suff} yields that $\Omega^{*}(|G_p|)\leq 1$. So $G_p\cong E_4$ by~(5) and hence $G\cong E_4\times C_{q^k}$ for some $k\geq 2$. Therefore $G$ belongs to the third family. If $|G_q|=q$ then from~(5) it follows that $G\cong C_{2q}\times C_{2^k}$ or $G\cong E_4 \times C_{2q}$. In the former case $G$ belongs to the second family; in the latter case $G$ belongs to the fourth family.

Now suppose that $p=3$. If $q\neq 2$ then $|G_q|=q$ because otherwise  $\Omega^{*}(|G_p|)\geq 2$ and $\Omega^{*}(|G_q|)\geq 2$ and $G$ is not weakly separable by Proposition~\ref{suff}. Proposition~\ref{suff} and~(5) imply that $G_p\cong C_3\times C_3$. So $G\cong C_3\times C_3 \times C_q$ and hence $G$ belongs to the seventh family. Let $q=2$. Then $|G_q|\leq 4$ by Proposition~\ref{suff}. If $|G_q|=2$ then in view of Proposition~\ref{suff} and~(5) the group $G$ is isomorphic to $C_6\times C_{3^k}$ or $C_6\times C_3\times C_3$. In the former case $G$ belongs to the sixth family; in the latter case $G$ belongs to the eighth family. If $|G_q|=4$ then from Proposition~\ref{suff} and~(5) it follows that $|G_p|=9$ and hence $G_p\cong C_3\times C_3$. Therefore $G\cong C_3\times C_3\times C_4$ or  $G\cong C_3\times C_3\times C_2 \times C_2$. In the former case $G$ belongs to the eight family; in the latter case $G$ is not weakly separable by Lemma~\ref{nonsep}.

Let $\omega(n)=3$. Then there exist distinct primes $q\neq p$ and $r\neq p$ such that 
$$G=G_p\times G_q\times G_r.$$ 
From~(4) it follows that $p\in\{2,3\}$. Suppose that $p=2$. Since $\Omega^{*}(|G_q\times G_r|)\geq 2$, we conclude that $\Omega^{*}(|G_p|)\leq 1$ and hence $G_p\cong E_4$. Assume that $|G_q|\geq q^2$. Then $\Omega^{*}(|G_q|)\geq 2$ and $\Omega^{*}(|G_p\times G_r|)\geq 2$. So $G$ is not weakly separable by Proposition~\ref{suff}, a contradiction. Therefore $|G_q|=q$. The similar argument implies that $|G_r|=r$. Thus, $G\cong E_4\times C_{qr}$ and hence $G$ belongs to the fourth family. 

Consider the remaining case $p=3$. In this case $\Omega^{*}(|G_p|)\geq 2$ and $\Omega^{*}(|G_q\times G_r|)\geq 2$ and hence $G$ is not weakly separable by Proposition~\ref{suff}, a contradiction. Thus, this case is impossible. The theorem is proved.
\end{proof}

\end{document}